\documentclass[reqno,12pt]{amsart}
\usepackage{amsmath,amssymb,amsthm,mathrsfs}
\usepackage{epsfig,color}
\usepackage[latin1]{inputenc}
\usepackage[english]{babel}
\usepackage[numbers]{natbib}
\usepackage[colorlinks, citecolor=blue, linkcolor=red]{hyperref}
\usepackage{dsfont}

\usepackage{hyperref}

\numberwithin{equation}{section}

\def\H{\mathbb{H}}

\newtheorem{ackn}{Acknowledgments\!}

\def\00{{\bf 0}}

\def\RR{\mathds R}

\newcommand{\eps}{{\varepsilon}}

\newtheorem*{theorem*}{Theorem}
\newtheorem{theorem}{Theorem}[section]
\newtheorem{lemma}[theorem]{Lemma}
\newtheorem{proposition}[theorem]{Proposition}

\newtheorem{remark}[theorem]{Remark}

\begin{document}
    \title[A Liouville theorem in the Heisenberg group]{A Liouville theorem in the Heisenberg group}

  \date{}

\author{Giovanni Catino, Yanyan Li, Dario D. Monticelli, Alberto Roncoroni}

\address{G. Catino, Dipartimento di Matematica, Politecnico di Milano, Piazza Leonardo da Vinci 32, 20133, Milano, Italy.}
\email{giovanni.catino@polimi.it}

\address{Y.Y. Li, Department of Mathematics, Rutgers University, 110 Frelinghuysen Road, Piscataway, NJ, USA}
\email{yyli@math.rutgers.edu}

\address{D. Monticelli, Dipartimento di Matematica, Politecnico di Milano, Piazza Leonardo da Vinci 32, 20133, Milano, Italy.}
\email{dario.monticelli@polimi.it}

\address{A. Roncoroni, Dipartimento di Matematica, Politecnico di Milano, Piazza Leonardo da Vinci 32, 20133, Milano, Italy.}
\email{alberto.roncoroni@polimi.it}

\begin{abstract} In this paper we classify positive solutions to the critical semilinear elliptic  equation in $\mathbb{H}^n$. We prove that they are the Jerison-Lee's bubbles, provided $n=1$ or $n\geq 2$ and a suitable control at infinity holds.  The proofs are based on a classical Jerison-Lee's differential identity and on pointwise/integral estimates recently obtained for critical semilinear and quasilinear elliptic equations in $\mathbb{R}^n$. In particular, the result in $\mathbb{H}^1$ can be seen as the analogue of the celebrated Caffarelli-Gidas-Spruck classification theorem.
\end{abstract}

\maketitle

\begin{center}


\medskip


\end{center}

\

\

\section{Introduction}

In this paper,  we consider solutions to the following {\em critical semilinear}  elliptic equation
\begin{equation}\label{heis_crit}
-\Delta_{\mathbb{H}^n}u = 2n^2 u^{q^{\ast}}\, \quad \text{in } \mathbb{H}^n
\end{equation}
where $\mathbb{H}^n$ is the Heisenberg group,  $u$ is a smooth, real and positive function defined in $\mathbb{H}^n$,  $\Delta_{\mathbb{H}^n}u$ is the Heisenberg Laplacian (or sub-Lapacian) of $u$ (see the definition in Section \ref{prel}) and
$$
q^{\ast}:=\frac{Q+2}{Q-2}
$$
with $Q=2n+2$ the homogeneous dimension of $\mathbb{H}^n$.

Equation \eqref{heis_crit} has been deeply studied since it is connected with the CR Yamabe problem in $\mathbb{H}^n$ and with the CR Sobolev inequality. The CR Yamabe problem on $\mathbb{H}^n$ is the following: given $(\H^n,\Theta)$ the sub-Riemannian manifold with standard contact form $\Theta$, consider the conformal contact form $\tilde\Theta=u^{\frac{2}{n}}\Theta$ on $\mathbb{H}^n$, then the pseudo-Hermitian scalar curvature associated to $\tilde\Theta$ is a positive constant, $R\equiv 4n(n+1)$, if and only if $u$ solves equation \eqref{heis_crit}.  The CR Yamabe problem has been studied in \cite{JL_CR1,JL,JL_CR2} and has been partially solved in \cite{Gam1,Gam2,Wang}; we also refer to the recent papers \cite{ChCh,CMY} for further developments. Moreover, the number
$$
q^{\ast}+1=\frac{2Q}{Q-2}\, ,
$$
is the critical exponent for the CR Sobolev embedding (or Folland-Stein inequality \cite{FS}).  Thanks to the work \cite{JL} we know that there are (nontrivial) positive solutions of \eqref{heis_crit} given by
\begin{equation}\label{tal}
\mathcal{U}_{\lambda,\mu}(z,t):=\dfrac{C}{\vert t+i \vert z\vert^2 + z\cdot\mu + \lambda \vert^n} \, ,
\end{equation}
for some $\lambda\in\mathbb{C}$, $\mu\in\mathbb{C}^n$ such that $\mathrm{Im}(\lambda)>\frac{\vert\mu\vert^2}{4}$ and for some explicit $C=C(n,\lambda)>0$. The functions in \eqref{tal} are the only extremals of the Folland-Stein inequality in $\mathbb{H}^n$ and are usually called {\em Jerison-Lee's bubbles}. Moreover, in \cite{JL} the authors obtained that \eqref{tal} are the unique positive solutions of \eqref{heis_crit}, satisfying the finite energy assumption $u\in L^{\frac{2Q}{Q-2}}(\mathbb{H}^n)$. We also refer to \cite{GarVas} where the authors obtained a uniqueness result under the assumption of cylindrical symmetry on groups of Heisenberg type.

Since it is well known that all nonnegative solutions of \eqref{heis_crit} are either strictly positive or identically zero, as the sub-Laplacian on the Heisenberg group satisfies the strong maximum principle (see e.g. \cite{Bony}), we will only focus on positive solutions. We also recall that there exist infinitely many nonradial sign-changing solutions with finite energy to
$$
-\Delta_{\mathbb{H}^n}u = 2n^2 |u|^{q^{\ast}-1}u\, \quad \text{in } \mathbb{H}^n,
$$
as proved in \cite{MaMa}.

\

\noindent In the subcritical case, i.e.
$$
-\Delta_{\mathbb{H}^n}u = 2n^2 u^{q}\, \quad \text{in } \mathbb{H}^n
$$
where $1<q<q^\ast$ it is known that the only nonnegative solution is the trivial one (see \cite{OuMa}, and \cite{BDC,BP,Xu} for previous partial results).

\

The analogue of \eqref{heis_crit} in the Euclidean space is the so-called critical Laplace equation
\begin{equation}\label{crit}
-\Delta u=u^{2^\ast-1} \quad \text{ in } \mathbb{R}^n\, ,
\end{equation}
where $2^*=\frac{2n}{n-2}$ is the critical Sobolev exponent. Equation \eqref{crit} is related to the Yamabe problem in Riemannian geometry (see the survey \cite{LP}) and to the extremals in the Sobolev inequality (see the survey \cite{Ron}).  From \cite{Rod}, \cite{Aubin} and \cite{Talenti} we know that the following class of functions
\begin{equation} \label{tal_R}
\mathcal V_{\lambda,x_0} (x) := \left( \frac{\lambda\sqrt{n (n-2)}}{\lambda^2 + |x-x_0|^2} \right)^{\frac{n-2}{2}} \,, \quad \lambda >0 \,, \ x_0 \in \mathbb{R}^n \,,
\end{equation}
solve \eqref{crit}. Moreover, from the seminal paper \cite{CGS} (see also \cite{GNN,Obata} for previous important results) we know that \eqref{tal_R} are the only positive solutions to \eqref{crit} (see also \cite{ChenLi} and \cite{LiZhang}). The proof of the classification result is based on the technique of moving planes and on the Kelvin transform. An alternative proof, based on integral estimates that can be applied also in the Riemannian setting when the Ricci curvature is non-negative, has been  recently obtained in \cite{CaMo} when the dimension is $n=3$, and was extended to dimensions $n=4,5$  in \cite{Ou, Vet_bis} respectively.

In the subcritical case, i.e.
\begin{equation}\label{subcrit}
-\Delta u=u^{q} \quad \text{ in } \mathbb{R}^n\,
\end{equation}
with $1<q<2^*-1$, it is well-known that the only nonnegative solution is the trivial one (see \cite{GS}). Recently, also the critical $p-$Laplace equation has been considered, we refer the interested reader to \cite{CMR,CFR,Ou,Ron_BP,Sci,Vet,Vet_bis}.

Our main results are a classification of \emph{all} positive solutions
 to \eqref{heis_crit} in $\mathbb{H}^1$ (see Theorem \ref{teo1}), and a classification of positive solutions to \eqref{heis_crit} in $\mathbb{H}^n$ when $n\geq2$
that satisfy a suitable decay condition at infinity, which is weaker than finite energy assumption (see Theorem \ref{teo2}). Indeed we have

\begin{theorem}\label{teo1}
 Let $u$ be a positive solution to \eqref{heis_crit} in $\H^1$. Then $$u \equiv \mathcal{U}_{\lambda,\mu}$$
 for some $\lambda\in\mathbb{C}$, $\mu\in\mathbb{C}^n$ such that $\mathrm{Im}(\lambda)>\frac{\vert\mu\vert^2}{4}$.
\end{theorem}

%(CMR):(CGS)=(JL):(OGNN)

\begin{theorem}\label{teo2}
Let $u$ be a positive solution to \eqref{heis_crit} in $\H^n$, $n\geq 2$ such that
$$
u(\xi)\leq\frac{C}{1+|\xi|^\frac{Q-2}{2}}\quad\forall\xi\in \H^n,
$$
for some $C>0$. Then $$u \equiv \mathcal{U}_{\lambda,\mu}$$
for some $\lambda\in\mathbb{C}$, $\mu\in\mathbb{C}^n$ such that $\mathrm{Im}(\lambda)>\frac{\vert\mu\vert^2}{4}$.
\end{theorem}

The proof of our results rely on a remarkable differential identity proved in \cite{JL}, which involves a vector field depending on the solution $u$ and its derivatives, that has nonnegative divergence whose vanishing implies that $u$ is actually one of the Jerison-Lee bubbles \eqref{tal}. Inspired by \cite{CaMo, CMR}, through a test function argument we are able to obtain integral estimates which, under the conditions stated in the theorems, imply that such divergence must vanish identically, thus giving the desired classification result. In order to obtain Theorem \ref{teo1} we need to suitably adapt the technique used in \cite{Ou, Vet_bis} to our setting; this allows us to obtain the full classification result without any extra assumption when $n=1$. In the proof of Theorem \ref{teo2}, when $n\geq2$, we also need an assumption on the behavior of the solution at infinity, which implies the validity of a useful gradient estimate on the solution $u$, that gives us the desired decay in the integral estimates.

We expect that the analogues of Theorems \ref{teo1} and \ref{teo2} should hold and yield a classification result for positive solutions of the critical sub-Laplace equation also in the Sasakian setting and in the context of Carnot groups, under appropriate geometric conditions (such as nonnegative pseudo-Hermitian Ricci curvature).

\

\noindent{\bf Organization of the paper.} In Section \ref{prel} we collect some preliminaries and notations, in Section \ref{proof_teo1} and \ref{proof_teo2} we prove Theorem \ref{teo1} and Theorem \ref{teo2}, respectively.  In Appendix \ref{grad_est} we prove a gradient estimate which is a key ingredient in the proof of Theorem \ref{teo2}.

\

\section{Preliminaries and Notations}\label{prel}

We first give a brief introduction to the Heisenberg group $\mathbb{H}^n$ with some notations (for further details we refer to \cite{book, JL_CR1, JL, Lee}). We consider
$$
\mathbb{H}^n:=\mathbb{C}^n\times\mathbb{R}
$$
with coordinates $\xi=(z,t)=(z_1,\dots,z_n,t)\in\mathbb{H}^n$ and with the group law $\circ$: given $\xi=(z,t)$ and $\zeta=(w,s)$
$$
(z,t)\circ(w,s)=\left(z+w,t+s+2\mathrm{Im} (z^\alpha\bar{w}^\alpha)\right) \, ,
$$
where here and in the sequel we use the Einstein notation sum for the Greek indices $1\leq\alpha,\beta,\gamma\leq n$. We define, for $\xi=(z,t)\in\mathbb{H}^n$, the norm
$$
\vert\xi\vert=\left( \vert z\vert^4+ t^2\right)^{\frac{1}{4}} \, ,
$$
with the associated distance function
$$
d(\xi,\zeta)=\vert \zeta^{-1}\circ \xi\vert \quad \text{for } \xi, \zeta\in\mathbb{H}^n\, ,
$$
where $\zeta^{-1}$ denotes the inverse of $\zeta$ with respect to $\circ$, i.e.  $\zeta^{-1}=-\zeta$.

We use the notation $B_R(\xi)$ for the metric ball centred at $\xi\in\mathbb{H}^n$ with radius $R>0$, i.e.
$$
B_R(\xi)=\lbrace \zeta\in\mathbb{H}^n \, : \, d(\xi, \zeta)<R\rbrace\,.
$$
If $\xi=0$, we will write $B_R:=B_R(0)$. It is well-known and it is important to recall that the volume of a metric ball is given by
\begin{equation}\label{volumi}
\vert B_r(\xi)\vert= C r^Q \, ,
\end{equation}
where $C>0$ is a positive constant, $Q=2n+2$ and $\vert\cdot\vert$ denotes the Lebesgue measure. The (even) integer $Q$ is called the homogeneous dimension of $\H^n$.

We define the following left-invariant (with respect to $\circ$) vector fields in $\mathbb{H}^n$
$$
Z_\alpha=\frac{\partial}{\partial z^\alpha}+i\bar{z}_\alpha\frac{\partial}{\partial t} \quad \text{and} \quad Z_{\bar{\alpha}}=\frac{\partial}{\partial \bar{z}^\alpha}-i z_\alpha\frac{\partial}{\partial t} \quad \text{for } \alpha=1,\dots,n.
$$
For a smooth function $f:\mathbb{H}^n\rightarrow\mathbb{C}$ we denote its derivatives by
$$
f_\alpha=Z_\alpha f \, ,\quad f_{\bar{\alpha}}=Z_{\bar{\alpha}} f\, , \quad f_0=\frac{\partial f}{\partial t} \, , \quad f_{\alpha\bar{\beta}}=Z_{\bar{\beta}}\left( Z_\alpha f\right) \, ,  \quad f_{0\alpha}=Z_{\alpha}\left(\frac{\partial f}{\partial t}\right)\, ,
$$
and so on.  There hold (see \cite{JL} and \cite{OuMa}) the following commutation rules
$$
f_{\alpha\beta}- f_{\beta\alpha}=0 \, , \quad f_{\alpha\bar{\beta}}- f_{\bar{\beta}\alpha}=2i\delta_{\alpha\bar{\beta}} f_0 \,, \quad f_{0\alpha}- f_{\alpha0}=0 \, ,
$$
$$
 f_{\alpha\beta 0}-f_{\alpha 0\beta}=0 \, ,\quad f_{\alpha\beta\bar{\gamma}}- f_{\alpha\bar{\gamma}\beta}=2i\delta_{\beta\bar{\gamma}}f_{\alpha0}\, .
$$
Moreover, we define
$$
\vert\partial f\vert^2:=\sum_{\alpha=1}^n f_\alpha f_{\bar{\alpha}}= f_\alpha f_{\bar{\alpha}} \quad \text{and} \quad \Delta_{\mathbb{H}^n}f:=\sum_{\alpha=1}^n \left( f_{\alpha\bar{\alpha}}+ f_{\bar{\alpha}\alpha}\right)=f_{\alpha\bar{\alpha}}+ f_{\bar{\alpha}\alpha} \, .	
$$

\

We recall that the Heisenberg group is a strictly pseudoconvex CR manifold, where the CR structure is given by the bundle $\mathcal{H}$ spanned by the vector fields $Z_\alpha$, for $\alpha=1,\ldots,n$, and where the standard contact form on $\mathbb{H}^n$ is given by
$$
\Theta=dt+\sum_{\alpha=1}^niz^\alpha d\bar{z}^\alpha-i\bar{z}^\alpha dz^{\alpha}.
$$

We also recall that the Heisenberg group is a Carnot group, which can be viewed as a flat model in Sub-Riemannian geometry similar to the Euclidean space $\mathbb{R}^n$ in Riemannian geometry, where the family of vector fields $T=\frac{\partial}{\partial t}$, $X_\alpha=2\operatorname{Re}Z_\alpha$, $Y_\alpha=2i\operatorname{Im}Z_\alpha$, for $\alpha=1,\ldots,n$, form a base of the Lie algebra of vector fields on $\mathbb{H}^n$ which are left
invariant with respect to the group action $\circ$.

\

Given $u>0$ a solution of \eqref{heis_crit} we consider the auxiliary function $f$ defined as follows
\begin{equation}\label{def_f}
e^f=u^{\frac{1}{n}} \, ,
\end{equation}
then $f$ solves
\begin{equation}\label{eq_f}
-\Delta	_{\mathbb{H}^n} f= 2n \vert\partial f\vert^2 + 2n e^{2f}\quad \text{in } \mathbb{H}^n\, .
\end{equation}
We also introduce the function $g:\mathbb{H}^n\rightarrow\mathbb{C}$ such that
\begin{equation}\label{g}
g=\vert\partial f\vert^2 + e^{2f} -if_0 \, ,
\end{equation}
then the equation \eqref{eq_f} can be rewritten as
\begin{equation}\label{eq_g}
f_{\alpha\bar\alpha}=-ng\quad 	\text{in } \mathbb{H}^n\, .
\end{equation}
As done in \cite{JL} and in \cite{OuMa} we define the following tensors
\begin{equation}\label{tens}
\begin{array}{lll}
& D_{\alpha\beta}=f_{\alpha\beta}-2f_\alpha f_\beta & D_\alpha=D_{\alpha\beta}f_{\bar{\beta}} \\
& E_{\alpha\bar\beta}=f_{\alpha\bar\beta}-\frac{1}{n}f_{\gamma\bar\gamma}\delta_{\alpha\bar\beta} \quad & E_\alpha=E_{\alpha\bar\beta}f_{\beta} \\
&G_\alpha=if_{0\alpha}-if_0f_\alpha+ e^{2f}f_\alpha+\vert\partial f\vert^2 f_{\alpha}\, .
\end{array}
\end{equation}
The above tensors will be important in our argument and we refer to \cite{JL} for the reason to introduce them.
Moreover (see also \cite{OuMa}) we observe that
\begin{equation}\label{tensg}
\begin{array}{lll}
& E_{\alpha\bar\beta}=f_{\alpha\bar\beta}+g\delta_{\alpha\bar\beta}  & E_\alpha=f_{\alpha\bar\beta}f_\beta+gf_\alpha\\
&D_\alpha=f_{\alpha\beta}f_{\bar\beta}-2|\partial f|^2f_\alpha & G_\alpha=if_{0\alpha}+gf_\alpha\\
&\vert\partial f\vert^2_{\bar{\alpha}}=D_{\bar{\alpha}}+E_{\bar{\alpha}}+ \bar{g} f_{\bar{\alpha}}-2 f_{\bar{\alpha}}e^{2f} \\
&g_{\bar\alpha}=D_{\bar\alpha}+E_{\bar\alpha}+G_{\bar\alpha} & \bar g_\alpha=D_{\alpha}+E_{\alpha}+G_{\alpha}.
\end{array}
\end{equation}
We are now in a position to recall the following differential identity obtained in \cite[Formula (4.2)]{JL} and \cite[Proposition 2.1]{OuMa} (with $p=0$) which will be fundamental in our arguments.
\begin{proposition}\label{prop_JL} With the notations above, we have
$$
\mathcal{M}=\mathrm{Re}Z_{\bar\alpha}\left\lbrace e^{2(n-1)f}\left[ \left( g+3if_0\right) E_\alpha + \left( g-if_0\right)D_\alpha -3i f_0 G_\alpha\right] \right\rbrace\, ,
$$
where
\begin{multline*}
\mathcal{M}= e^{2nf}\left( \vert E_{\alpha\bar{\beta}}\vert^2 + \vert D_{\alpha\beta}\vert^2\right)   \\
+ e^{2(n-1)f}\left( \vert G_\alpha \vert^2 + \vert G_\alpha+ D_\alpha\vert^2 + \vert G_\alpha -E_\alpha\vert^2 + \vert D_{\alpha\beta}f_{\bar{\gamma}}+ E_{\alpha\bar{\gamma}}f_\beta\vert^2\right) \, .
\end{multline*}
\end{proposition}
From this proposition we obtain the next two lemmas: the first one in the case $n=1$ and the second one for $n\geq 2$.

\begin{lemma}\label{lemma0}
With the notations above, if $n=1$, then for every (real) non-negative cut-off function $\eta$ with compact support and for every $s>2$ and $\beta>0$ small enough we have
$$
\int_{\mathbb{H}^1} \mathcal{M}\Psi^{-\beta}\eta^s\leq  C \left(\int_{\text{supp}|\partial\eta|}\mathcal{M}\Psi^{-\beta}\eta^s\right)^{1/2}\left(\int_{\text{supp}|\partial\eta|} |g|^2\Psi^{-\beta}|\partial\eta|^2\eta^{s-2}\right)^{1/2}\,,
$$
for some $C>0$, where
$$
\Psi:=g\bar{g}\, e^{-2f}=\vert g\vert^2 \, e^{-2f}\, .
$$
In particular
$$
\int_{\mathbb{H}^1} \mathcal{M}\Psi^{-\beta}\eta^s\leq  C \int_{\mathbb{H}^1}\vert g\vert^2 \Psi^{-\beta} \vert\partial\eta\vert^2 \eta^{s-2}\,.
$$
\end{lemma}

\begin{proof}
We define
$$
\mathcal{I}_1:=\int_{\mathbb{H}^1} \mathcal{M}\Psi^{-\beta}\eta^s\,.
$$
From Proposition \ref{prop_JL} we obtain
\begin{align}\label{eq0}
\int_{\mathbb{H}^1} \mathcal{M}\Psi^{-\beta}\eta^s=&\int_{\mathbb{H}^1}  \mathrm{Re}Z_{\bar\alpha} \left[ \left( g+3if_0\right) E_\alpha + \left( g-if_0\right)D_\alpha -3i f_0 G_\alpha\right]  \Psi^{-\beta} \eta^s \nonumber \\
=& \beta \int_{\mathbb{H}^1}  \mathrm{Re}\left\lbrace\left[ \left( g+3if_0\right) E_\alpha + \left( g-if_0\right)D_\alpha -3i f_0 G_\alpha\right]\Psi_{\bar{\alpha}}\right\rbrace \Psi^{-\beta-1} \eta^s \nonumber  \\
&-s \int_{\mathbb{H}^1} \mathrm{Re}\left\lbrace \left[ \left( g+3if_0\right) E_\alpha + \left( g-if_0\right)D_\alpha -3i f_0 G_\alpha\right]\eta_{\bar{\alpha}} \right\rbrace \Psi^{-\beta} \eta^{s-1}
\end{align}
where we integrate by parts. We now observe that, on the one hand, from the definition of $g$ \eqref{g} we have
\begin{align*}
\left( g+3if_0\right) E_\alpha + \left( g-if_0\right)D_\alpha -3i f_0 G_\alpha = &\,  \left( D_\alpha+ G_\alpha\right) (g-if_0) + \left( E_\alpha - G_\alpha\right) (g+3if_0)+ if_0G_\alpha \\
=& \left( D_\alpha+ G_\alpha\right) \left( \vert\partial f\vert^2 + e^{2f} -2i f_0\right)\\
&\, + \left( E_\alpha - G_\alpha\right) \left( \vert\partial f\vert^2 + e^{2f} +2if_0\right)+ if_0G_\alpha    \, ,
\end{align*}
and so, from Cauchy-Schwarz inequality
\begin{align*}
\vert \left( g+3if_0\right) E_\alpha + \left( g-if_0\right)D_\alpha -3i f_0 G_\alpha \vert\leq & \left\vert D_\alpha+ G_\alpha\right\vert \sqrt{\vert\partial f\vert^4 + e^{4f}+2\vert\partial f\vert^2 e^{2f} +4 f_0^2}\\
&\, + \left\vert E_\alpha - G_\alpha\right\vert \sqrt{\vert\partial f\vert^4 + e^{4f}+2\vert\partial f\vert^2 e^{2f} +4 f_0^2}\\
&\,+ \vert f_0\vert\vert G_\alpha\vert    \\
 \leq &\,  2\vert g\vert \left( \left\vert D_\alpha+ G_\alpha\right\vert +\left\vert E_\alpha - G_\alpha\right\vert  + \vert G_\alpha\vert \right)\\
 \leq &\, 2\vert g\vert\sqrt{\mathcal{M}}
\end{align*}
where we used the fact that
\begin{equation}\label{mod_g}
\vert g\vert=\sqrt{\vert\partial f\vert^4 + e^{4f}+2\vert\partial f\vert^2 e^{2f} + f_0^2}\, .
\end{equation}
Summing up, we have obtained the following
\begin{equation}\label{p1}
\vert \left( g+3if_0\right) E_\alpha + \left( g-if_0\right)D_\alpha -3i f_0 G_\alpha\vert \leq  \,  2\vert g\vert\sqrt{\mathcal{M}}\, .
\end{equation}
On the other hand, from \eqref{tensg} we have
\begin{align*}
\Psi_{\bar{\alpha}}=&\left[ \left( g\bar{g}\right) e^{-2f}\right]_{\bar{\alpha}}=e^{-2f} \left( \bar{g}g_{\bar{\alpha}} + g \bar{g}_{\bar{\alpha}}\right) - 2\left( g\bar{g}\right)f_{\bar{\alpha}}e^{-2f} \\
=& \,  e^{-2f} \left[ \bar{g}\left( D_{\bar\alpha}+E_{\bar\alpha}+G_{\bar\alpha}\right) + g \left( D_{\bar\alpha}+E_{\bar\alpha}-G_{\bar\alpha}+2\bar{g}f_{\bar{\alpha}}\right)\right] -2\left( g\bar{g}\right) f_{\bar{\alpha}}e^{-2f} \\
=\, & e^{-2f} \left[ D_{\bar{\alpha}} \left( g + \bar{g}\right) + E_{\bar{\alpha}}\left( g + \bar{g}\right) + G_{\bar{\alpha}}\left(\bar{g}-g \right) \right] \\
=&\,   e^{-2f} \left[ \left( D_{\bar{\alpha}} + G_{\bar{\alpha}}\right) \left( g + \bar{g}\right) + \left( E_{\bar{\alpha}}-G_{\bar{\alpha}}\right) \left( g + \bar{g}\right) + G_{\bar{\alpha}}\left(\bar{g}-g \right) \right] \\
=& \,   2 e^{-2f} \left[ \left( D_{\bar{\alpha}} + G_{\bar{\alpha}}\right) \left( \vert\partial f\vert^2 + e^{2f} \right) + \left( E_{\bar{\alpha}}-G_{\bar{\alpha}}\right) \left( \vert\partial f\vert^2 + e^{2f} \right) + i G_{\bar{\alpha}}f_0 \right] \, ,
\end{align*}
where we used the fact that
\begin{equation}\label{gbarra}
\bar{g}_{\bar{\alpha}}= D_{\bar\alpha}+E_{\bar\alpha}-G_{\bar\alpha}+2\bar{g}f_{\bar{\alpha}} \, .
\end{equation}
Indeed by \eqref{tensg} we have
$$
G_{\bar{\alpha}}=-if_{0\bar{\alpha}}+\bar{g}f_{\bar{\alpha}}\,,
$$
hence
\begin{align*}
  \bar{g}_{\bar{\alpha}}= & \left(|\partial f|^2+e^{2f}+if_0\right)_{\bar{\alpha}} \\
   =& D_{\bar{\alpha}}+E_{\bar{\alpha}}+ \bar{g} f_{\bar{\alpha}}+if_{0\bar{\alpha}}\\
   =& D_{\bar\alpha}+E_{\bar\alpha}-G_{\bar\alpha}+2\bar{g}f_{\bar{\alpha}} \, .
\end{align*}
Moreover,  from Cauchy-Schwarz inequality
\begin{align*}
\vert\Psi_{\bar{\alpha}}\vert\leq & \,   2 e^{-2f} \left[ \left\vert D_{\bar{\alpha}} + G_{\bar{\alpha}}\right\vert \sqrt{\vert\partial f\vert^4 + e^{4f}+2\vert\partial f\vert^2 e^{2f}} \right.\\
&\, \left.+ \left\vert E_{\bar{\alpha}}-G_{\bar{\alpha}}\right\vert \sqrt{\vert\partial f\vert^4 + e^{4f}+2\vert\partial f\vert^2 e^{2f}} +  \vert G_{\bar{\alpha}}\vert \vert f_0\vert \right] \\
\leq & 2 e^{-2f} \vert g\vert  \left[ \vert D_{\bar{\alpha}} + G_{\bar{\alpha}}\vert + \vert E_{\bar{\alpha}}-G_{\bar{\alpha}}\vert + \vert G_{\bar{\alpha}}\vert  \right]  \\
\leq & \,  2 e^{-2f}\vert g\vert\sqrt{\mathcal{M}}\, ,
\end{align*}
i.e.
\begin{equation}\label{p2}
\vert \Psi_{\bar{\alpha}}\vert \leq  \,  2 e^{-2f}\vert g\vert\sqrt{\mathcal{M}}\, .
\end{equation}
Hence, by substituting \eqref{p1} and \eqref{p2} in \eqref{eq0} we get, after a Cauchy-Schwarz inequality,
\begin{align*}
\mathcal{I}_1= \left\vert \int_{\mathbb{H}^1} \mathcal{M}\Psi^{-\beta}\eta^s\right\vert \leq & 4\beta \int_{\mathbb{H}^1} \mathcal{M}\vert g\vert^2 e^{-2f}    \Psi^{-\beta-1} \eta^s \nonumber + 2s \int_{\mathbb{H}^1} \vert g\vert \sqrt{\mathcal{M}}  \Psi^{-\beta}\vert\partial \eta\vert \eta^{s-1}  \\
\leq & \,  4\beta \int_{\mathbb{H}^1} \mathcal{M}   \Psi^{-\beta} \eta^s + 2s \int_{\mathbb{H}^1} \sqrt{\mathcal{M}} \vert g\vert\Psi^{-\beta}\vert\partial \eta\vert \eta^{s-1}\,.
\end{align*}
Choosing $\beta>0$ small enough we find
\begin{equation*}
  \mathcal{I}_1\leq C\int_{\mathbb{H}^1} \sqrt{\mathcal{M}} \vert g\vert\Psi^{-\beta}\vert\partial \eta\vert \eta^{s-1}\,.
\end{equation*}
We now use H\"{o}lder's inequality and we have
$$
\mathcal{I}_1\leq C \left(\int_{\mathrm{supp}|\partial\eta|} \mathcal{M}\Psi^{-\beta}\eta^s\right)^\frac{1}{2}\left( \int_{\mathrm{supp}|\partial\eta|} \vert g\vert^2 \Psi^{-\beta}\vert\partial\eta\vert^2\eta^{s-2} \right)^\frac{1}{2} \, ,
$$
for some $C>0$. Now the conclusion easily follows.
\end{proof}

\begin{lemma}\label{lemma1}
With the notations above, for every (real) non-negative cut-off function $\eta$ with compact support and for every $s>2$ we have
$$
\int_{\mathbb{H}^n} \mathcal{M}\eta^s\leq  C \left(\int_{\text{supp}|\partial\eta|}\mathcal{M}\eta^s\right)^{1/2}\left(\int_{\text{supp}|\partial\eta|} e^{2(n-1)f}|g|^2|\partial\eta|^2\eta^{s-2}\right)^{1/2}\,.
$$
In particular we also have
$$
\int_{\mathbb{H}^n} \mathcal{M}\eta^s\leq  C \int_{\mathbb{H}^n}\vert g\vert^2 e^{2(n-1)f} \vert\partial\eta\vert^2 \eta^{s-2}\, .
$$
\end{lemma}
\begin{proof} From Proposition \ref{prop_JL},  integrating by parts, using Cauchy-Schwarz and Young inequalities we obtain
\begin{align*}
&\int_{\H^n}\mathcal{M}\eta^s=\int_{\H^n}\mathrm{Re}Z_{\bar\alpha}\left\lbrace e^{2(n-1)f}\left[ \left( g+3if_0\right) E_\alpha + \left( g-if_0\right)D_\alpha -3i f_0 G_\alpha\right] \right\rbrace \eta^s\\
&=-s\int_{\H^n}\mathrm{Re}\left\lbrace e^{2(n-1)f}\left[ \left( g+3if_0\right) E_\alpha + \left( g-if_0\right)D_\alpha -3i f_0 G_\alpha\right] \eta_{\bar\alpha}\right\rbrace \eta^{s-1}\\
&=- s\int_{\H^n}\mathrm{Re}\left\lbrace e^{2(n-1)f}\left[ \left( \vert\partial f\vert^2 + e^{2f} +2if_0\right) E_\alpha + \left( \vert\partial f\vert^2 + e^{2f} -2if_0\right)D_\alpha -3i f_0 G_\alpha\right] \eta_{\bar{\alpha}}\right\rbrace \eta^{s-1}\\
&\leq s\int_{\H^n} e^{2(n-1)f}\left[ \left( \vert\partial f\vert^2 + e^{2f} +2|f_0|\right) |E_\alpha| + \left( \vert\partial f\vert^2 + e^{2f} +2|f_0|\right)|D_\alpha| +3 |f_0| |G_\alpha|\right] |\partial\eta|\eta^{s-1}\\
&\leq C\left( \int_{\mathrm{supp}|\partial\eta|} e^{2(n-1)f}\left[ |E_\alpha|^2+|D_\alpha|^2 +|G_\alpha|^2\right]\eta^s\right)^\frac{1}{2}\\
&\qquad\qquad \left(\int_{\mathrm{supp}|\partial\eta|} e^{2(n-1)f}\left[|\partial f|^4+e^{4f}+|f_0|^2\right]|\partial\eta|^2\eta^{s-2}\right)^\frac{1}{2}\\
&\leq C\left( \int_{\mathrm{supp}|\partial\eta|} e^{2(n-1)f}\left[ |G_\alpha-E_\alpha|^2+|G_\alpha+D_\alpha|^2 +|G_\alpha|^2\right]\eta^s\right)^\frac{1}{2}\\
&\qquad\qquad  \left( \int_{\mathrm{supp}|\partial\eta|} e^{2(n-1)f}|g|^2|\partial\eta|^2\eta^{s-2}\right)^\frac{1}{2}\\
&\leq C\left( \int_{\mathrm{supp}|\partial\eta|} \mathcal{M}\eta^s\right)^\frac{1}{2}\left(\int_{\mathrm{supp}|\partial\eta|} e^{2(n-1)f}|g|^2|\partial\eta|^2\eta^{s-2}\right)^\frac{1}{2}
\end{align*}
which immediately yields the conclusion.
\end{proof}

We conclude this section by recalling the following lower bound for positive superhamonic functions in $\H^n$ (see \cite{bir}).
\begin{proposition}\label{stimau_basso}
Let $u$ be a positive superharmonic function in $\H^n$, i.e $u\in C^2(\H^n)$ and
$$\Delta_{\H^n}u\leq 0\quad\text{in }\H^n.$$
Then, there exists a constant $C>0$ such that
$$
u(\xi) \geq \frac{C}{|\xi|^{Q-2}},
$$
for any $\xi\in\H^n$ with $|\xi|>1$, where $Q=2n+2$.
\end{proposition}

\

\section{Proof of Theorem \ref{teo1}}\label{proof_teo1}

In this section $n=1$ and given $R>0$, we choose a real cut-off function $\eta$ such that $\eta\equiv 1$ in $B_{R/2}$, $\eta\equiv 0$ in $B_R^c$ and $\vert \partial\eta\vert\leq \frac{c}{R}$ in $B_R\setminus B_{R/2}$. Let $s>6$ and $\beta>0$ small enough.

From Lemma \ref{lemma0} we have
\begin{align}\label{eq13}\nonumber
\mathcal{I}_1=\int_{\mathbb{H}^1} \mathcal{M}\Psi^{-\beta}\eta^s\leq & C \int_{\mathbb{H}^1}\vert g\vert^2\Psi^{-\beta} \vert\partial\eta\vert^2 \eta^{s-2} \\
\leq& \, \dfrac{C}{R^2} \int_{B_R\setminus B_{R/2}} \vert g\vert^2\Psi^{-\beta} \eta^{s-2} \\ \nonumber
\leq &  \, \dfrac{C}{R^2} \int_{B_R} \left(g\bar{g}\right)\Psi^{-\beta}\eta^{s-2} \\ \nonumber
=& -\dfrac{C}{R^2} \int_{B_R} \mathrm{Re} (f_{\alpha\bar{\alpha}}\bar{g}) \Psi^{-\beta} \eta^{s-2}\\\label{eqifinal}
=&:\dfrac{C}{R^2} \mathcal{I}_2\, ,
\end{align}
where we used \eqref{eq_g}.  By integrating by parts and using \eqref{gbarra} we get
\begin{align*}
\mathcal{I}_2=&- \int_{B_R} \mathrm{Re} (f_{\alpha\bar{\alpha}}\bar{g}) \Psi^{-\beta} \eta^{s-2}\\
=& \int_{B_R} \mathrm{Re} (f_\alpha \bar{g}_{\bar\alpha})   \Psi^{-\beta} \eta^{s-2} -\beta  \int_{B_R} \mathrm{Re} (f_\alpha \Psi_{\bar{\alpha}}\bar{g})\Psi^{-\beta-1}\eta^{s-2}\\
& + (s-2) \int_{B_R} \mathrm{Re} (f_\alpha\bar{g}\eta_{\bar{\alpha}})\Psi^{-\beta}\eta^{s-3} \\
=& \,  \int_{B_R} \mathrm{Re} (f_\alpha \left( D_{\bar\alpha}+E_{\bar\alpha}-G_{\bar\alpha}+2\bar{g}f_{\bar{\alpha}}\right))   \Psi^{-\beta} \eta^{s-2} -\beta  \int_{B_R} \mathrm{Re} (f_\alpha \Psi_{\bar{\alpha}}\bar{g})\Psi^{-\beta-1}\eta^{s-2}\\
& + (s-2) \int_{B_R} \mathrm{Re} (f_\alpha\bar{g}\eta_{\bar{\alpha}})\Psi^{-\beta}\eta^{s-3} \\
=& \int_{B_R} \mathrm{Re} (f_\alpha \left( D_{\bar\alpha}+E_{\bar\alpha}-G_{\bar\alpha}\right))   \Psi^{-\beta} \eta^{s-2} + 2 \int_{B_R} \mathrm{Re}(\bar{g})\vert\partial f\vert^2 \Psi^{-\beta}\eta^{s-2} \\
&-\beta  \int_{B_R} \mathrm{Re} (f_\alpha \Psi_{\bar{\alpha}}\bar{g})\Psi^{-\beta-1}\eta^{s-2} + C \int_{B_R} \vert\partial f\vert \vert g\vert \vert \partial\eta\vert\Psi^{-\beta}\eta^{s-3}.
\end{align*}
Since
\begin{equation}\label{stimaEFG}
\vert D_{\bar\alpha}+E_{\bar\alpha}-G_{\bar\alpha}\vert =\vert  (D_{\bar\alpha}+G_{\bar\alpha})+(E_{\bar\alpha}-G_{\bar\alpha})-G_{\bar\alpha}\vert\leq \sqrt\mathcal{M},
\end{equation}
for every $\theta>0$, using \eqref{p2}, we obtain
\begin{align*}
\mathcal{I}_2\leq&\theta \int_{B_R} \mathcal M\Psi^{-\beta}\eta^s+\frac{1}{4\theta}\int_{B_R} |\partial f|^2\Psi^{-\beta}\eta^{s-4}+2\int_{B_R}\left(|\partial f|^4+e^{2f}|\partial f|^2\right)\Psi^{-\beta}\eta^{s-2}\\
&+2\beta  \int_{B_R} |\partial f|\sqrt\mathcal{M}\Psi^{-\beta}\eta^{s-2}+ \frac{C_\eps}{R^2} \int_{B_R} |\partial f|^2\Psi^{-\beta}\eta^{s-4}+\eps\mathcal{I}_2\\
\leq&\left(\theta+C\beta R^2\right) \int_{B_R} \mathcal M\Psi^{-\beta}\eta^s+C\left(\frac{1}{\theta}+\frac{1}{R^2}+\frac{\beta}{R^2}\right)\int_{B_R} |\partial f|^2\Psi^{-\beta}\eta^{s-4}\\
&+4\int_{B_R}|\partial f|^4\Psi^{-\beta}\eta^{s-2}+\frac12 \int_{B_R}e^{4f}\Psi^{-\beta}\eta^{s-2}+\eps\mathcal{I}_2,
\end{align*}
i.e., by choosing $\eps$ small enough,
\begin{align*}
\mathcal{I}_2\leq&C\left(\theta+\beta R^2\right) \int_{B_R} \mathcal M\Psi^{-\beta}\eta^s+C\left(\frac{1}{\theta}+\frac{1}{R^2}+\frac{\beta}{R^2}\right)\int_{B_R} |\partial f|^2\Psi^{-\beta}\eta^{s-4}\\
&+C\int_{B_R}|\partial f|^4\Psi^{-\beta}\eta^{s-2}+\frac23 \int_{B_R}e^{4f}\Psi^{-\beta}\eta^{s-2}.
\end{align*}
Since
$$
\mathcal{I}_2 = \int_{B_R}\left(|\partial f|^4+e^{4f}+2e^{2f}|\partial f|^2+f_0^2\right)\Psi^{-\beta}\eta^{s-2},
$$
we obtain
\begin{align}\label{eqi2}\nonumber
\mathcal{I}_2\leq&C(\eps+\beta) R^2 \int_{B_R} \mathcal M\Psi^{-\beta}\eta^s+\frac{C(\eps^{-1}+\beta)}{R^2}\int_{B_R} |\partial f|^2\Psi^{-\beta}\eta^{s-4}+C\int_{B_R}|\partial f|^4\Psi^{-\beta}\eta^{s-2}\\
=:&C(\eps+\beta) R^2 \mathcal{I}_1+\frac{C(\eps^{-1}+\beta)}{R^2}\mathcal{I}_3+C\mathcal{I}_4,
\end{align}
where we chose $\theta=\eps R^2$.

Now, since
\begin{equation}\label{trick}
\Psi^{-\beta}= \vert g\vert^{-2\beta} e^{2\beta f} \leq e^{2\beta f}\left( \vert\partial f\vert^4 + e^{4f}\right)^{-\beta} \leq e^{-2\beta f}
\end{equation}
where we used \eqref{mod_g}, then using once again \eqref{eq_g} and integration by parts
\begin{align}\label{eqii3}
\mathcal{I}_3=& \, \int_{B_R}  |\partial f|^2\Psi^{-\beta}\eta^{s-4}\nonumber \\
\leq & \,  \int_{B_R} |\partial f|^2 e^{-2\beta f} \eta^{s-4} =:  \, \mathcal{I}_3' \\ \label{eqi3}
\leq & \,   \int_{B_R} \mathrm{Re}(g)e^{-2\beta f} \eta^{s-4}\nonumber \\
=& \, - \int_{B_R} \mathrm{Re}(f_{\alpha\bar{\alpha}})e^{-2\beta f} \eta^{s-4}\nonumber \\
=& \,  -2\beta \int_{B_R} \vert\partial f\vert^2 e^{-2\beta f}\eta^{s-4} + (s-4)\int_{B_R} \mathrm{Re}(f_\alpha\eta_{\bar{\alpha}})e^{-2\beta f} \eta^{s-5}\nonumber\\
\leq & \,  -2\beta\mathcal{I}_3' + \varepsilon'\mathcal{I}_3' + \frac{C_{\varepsilon'}}{R^2}\int_{B_R} e^{-2\beta f}
\end{align}
for every $\varepsilon'>0$ and for some $C_{\varepsilon'}>0$, where we used the following Young inequality
$$
\mathrm{Re}(f_\alpha\eta_{\bar{\alpha}})e^{-2\beta f} \eta^{s-5}\leq \vert\partial f\vert\vert\partial\eta\vert e^{-2\beta f} \eta^{s-5}\leq \varepsilon' \vert\partial f\vert^2 e^{-2\beta f} \eta^{s-4} + C_{\varepsilon'} \vert\partial\eta\vert^2 e^{-2\beta f}\eta^{s-6}  \, .
$$
Now, we observe that from \eqref{def_f} and from Proposition \ref{stimau_basso} we obtain
\begin{equation}\label{stima_f}
e^{-2\beta f}= u^{-2\beta} \leq C R^{4\beta},
\end{equation}
for $R$ large enough. Hence, by choosing $\varepsilon'=2\beta$ in \eqref{eqi3} we find
\begin{equation}\label{eqi8}
\mathcal{I}_3\leq\mathcal{I}_3' \leq C R ^{2+4\beta}\, ,
\end{equation}
where we used \eqref{stima_f}, \eqref{eqii3} and \eqref{volumi}.

We argue in the same way in order to estimate $\mathcal{I}_4$; from 	\eqref{trick}, \eqref{g}, \eqref{eq_g} and integration by parts we get
\begin{align}\label{eqi7}
\mathcal{I}_4=& \int_{B_R}|\partial f|^4\Psi^{-\beta}\eta^{s-2}\nonumber\\
\leq& \, \int_{B_R} | \partial f|^4 e^{-2\beta f} \eta^{s-2}=: \mathcal{I}_4' \\
\leq & \,  \int_{B_R} | \partial f|^2\mathrm{Re}(g) e^{-2\beta f}\eta^{s-2} \nonumber\\
=& \,  -\int_{B_R} | \partial f|^2\mathrm{Re}(f_{\alpha\bar{\alpha}}) e^{-2\beta f}\eta^{s-2} \nonumber\\
=& \int_{B_R} \mathrm{Re}(f_\alpha\vert\partial f\vert^2_{\bar{\alpha}}) e^{-2\beta f}\eta^{s-2} - 2\beta \int_{B_R} \vert\partial f\vert^4 e^{-2\beta f}\eta^{s-2}\nonumber \\
&+ (s-2) \int_{B_R} \mathrm{Re}(f_\alpha\eta_{\bar{\alpha}})\vert\partial f\vert^2 e^{-2\beta f}\eta^{s-3} \nonumber\\
=& \int_{B_R} \mathrm{Re}(f_\alpha (D_{\bar{\alpha}}+ E_{\bar{\alpha}}) ) e^{-2\beta f}\eta^{s-2} + \int_{B_R}\vert\partial f\vert^2 \mathrm{Re}(\bar{g}) e^{-2\beta f}\eta^{s-2} -2 \int_{B_R}  \vert\partial f\vert^2 e^{2(1-\beta) f}\eta^{s-2}\nonumber \\
& - 2\beta \int_{B_R} \vert\partial f\vert^4 e^{-2\beta f}\eta^{s-2} + (s-2) \int_{B_R} \mathrm{Re}(f_\alpha\eta_{\bar{\alpha}})\vert\partial f\vert^2 e^{-2\beta f}\eta^{s-3}\nonumber \, ,
\end{align}
where we used \eqref{tensg}.  Summing up
\begin{align*}
\mathcal{I'}_4 \leq &  \int_{B_R} \mathrm{Re}(f_\alpha (D_{\bar{\alpha}}+ E_{\bar{\alpha}}) ) e^{-2\beta f}\eta^{s-2} + \int_{B_R}\vert\partial f\vert^4  e^{-2\beta f}\eta^{s-2} - \int_{B_R}  \vert\partial f\vert^2 e^{2(1-\beta) f}\eta^{s-2} \\
& - 2\beta \int_{B_R} \vert\partial f\vert^4 e^{-2\beta f}\eta^{s-2} + (s-2) \int_{B_R} \mathrm{Re}(f_\alpha\eta_{\bar{\alpha}})\vert\partial f\vert^2 e^{-2\beta f}\eta^{s-3} \\
=& \int_{B_R} \mathrm{Re}(f_\alpha (D_{\bar{\alpha}}+ E_{\bar{\alpha}}) ) e^{-2\beta f}\eta^{s-2} +\mathcal{I}_4' - \int_{B_R}  \vert\partial f\vert^2 e^{2(1-\beta) f}\eta^{s-2} \\
& - 2\beta \mathcal{I}_4' + (s-2) \int_{B_R} \mathrm{Re}(f_\alpha\eta_{\bar{\alpha}})\vert\partial f\vert^2 e^{-2\beta f}\eta^{s-3}
\end{align*}
i.e.
\begin{align}\label{eqi4}
2\beta \mathcal{I'}_4 \leq &   \int_{B_R} \mathrm{Re}(f_\alpha (D_{\bar{\alpha}}+ E_{\bar{\alpha}}) ) e^{-2\beta f}\eta^{s-2} - \int_{B_R}  \vert\partial f\vert^2 e^{2(1-\beta) f}\eta^{s-2}\nonumber \\
& + (s-2) \int_{B_R} \mathrm{Re}(f_\alpha\eta_{\bar{\alpha}})\vert\partial f\vert^2 e^{-2\beta f}\eta^{s-3} \, .
\end{align}
Now, we tackle the first integral on the right-hand side of the \eqref{eqi4}
\begin{align*}
 \int_{B_R} \mathrm{Re}(f_\alpha (D_{\bar{\alpha}}+ E_{\bar{\alpha}}) ) e^{-2\beta f}\eta^{s-2} \leq & \int_{B_R} \sqrt{\mathcal{M}} \vert\partial f\vert  e^{-2\beta f}\eta^{s-2} \\
 = & \,  \int_{B_R} \sqrt{\mathcal{M}} \vert\partial f\vert \vert g\vert^{2\beta}  e^{-4\beta f}\Psi^{-\beta} \eta^{s-2}
\end{align*}
where we used \eqref{stimaEFG}. By using Young inequality,
\begin{equation*}
 \int_{B_R} \mathrm{Re}(f_\alpha (D_{\bar{\alpha}}+ E_{\bar{\alpha}}) ) e^{-2\beta f}\eta^{s-2} \leq  \varepsilon R^2\mathcal{I}_1 + \frac{C}{R^2}\int_{B_R}  \vert\partial f\vert^2 \vert g\vert^{4\beta}  e^{-8\beta f}\Psi^{-\beta} \eta^{s-4}
\end{equation*}
for every $\varepsilon>0$ and for some $C>0$. From \eqref{stima_f} we obtain
\begin{align*}
 \int_{B_R} \mathrm{Re}(f_\alpha (D_{\bar{\alpha}}+ E_{\bar{\alpha}}) ) e^{-2\beta f}\eta^{s-2} \leq &  \varepsilon R^2\mathcal{I}_1 + \frac{C}{R^{2}}\mathcal{I}_2+\frac{C}{R^{2}}\int_{B_R} \vert\partial f\vert^{\frac{2}{1-2\beta}}e^{-\frac{8\beta f}{1-2\beta}}\Psi^{-\beta}\eta^{s-\frac{4-4\beta}{1-2\beta}}   \, ,
\end{align*}
where we used the following Young inequality
$$
 \vert\partial f\vert^2 \vert g\vert^{4\beta}e^{-8\beta f}  \eta^{s-4} \leq 2\beta \vert g\vert^2 \eta^{s-2} + (1-2\beta) \vert\partial f\vert^{\frac{2}{1-2\beta}}e^{-\frac{8\beta f}{1-2\beta}}\eta^{s-\frac{4-4\beta}{1-2\beta}}\, .
$$
Hence,
\begin{align*}
 \int_{B_R} \mathrm{Re}(f_\alpha (D_{\bar{\alpha}}+ E_{\bar{\alpha}}) ) e^{-2\beta f}\eta^{s-2}  \leq & \, \varepsilon R^2\mathcal{I}_1 + \frac{C}{R^{2}}\mathcal{I}_2+\frac{C}{R^{2}}\int_{B_R} \vert\partial f\vert^{\frac{2}{1-2\beta}}e^{-\frac{8\beta f}{1-2\beta}}\Psi^{-\beta}\eta^{s-\frac{4-4\beta}{1-2\beta}}  \\
 \leq & \, \varepsilon R^2\mathcal{I}_1 + \frac{C}{R^{2}}\mathcal{I}_2+\frac{C}{R^{2}}\mathcal{I}_3  +
 \frac{C}{R^{2-36\beta}}\mathcal{I}_4+\frac{C}{R^2}\int_{B_R}\eta^{s-\frac{6-8\beta}{1-4\beta}}\, ,
\end{align*}
where we used the following Young inequality
$$
\vert\partial f\vert^{\frac{2}{1-2\beta}}e^{-\frac{8\beta f}{1-2\beta}}\Psi^{-\beta}\eta^{s-\frac{4-4\beta}{1-2\beta}}\leq \frac{1}{2(1-2\beta)}\vert\partial f\vert^4e^{-16\beta f}\Psi^{-2\beta}\eta^{s-2} + \frac{1-4\beta}{2-4\beta}\eta^{s-\frac{6-8\beta}{1-4\beta}}\,
$$
together with \eqref{trick} and \eqref{stima_f}. From \eqref{volumi} we obtain
\begin{equation}\label{eqi5}
 \int_{B_R} \mathrm{Re}(f_\alpha (D_{\bar{\alpha}}+ E_{\bar{\alpha}}) ) e^{-2\beta f}\eta^{s-2}  \leq \varepsilon R^2\mathcal{I}_1 + \frac{C}{R^{2}}\mathcal{I}_2+\frac{C}{R^{2}}\mathcal{I}_3  +
 \frac{C}{R^{2-36\beta}}\mathcal{I}_4+CR^2\, ,
\end{equation}
By using \eqref{eqi5} in \eqref{eqi4} we get
\begin{align}\label{eqi6}
2\beta \mathcal{I'}_4 \leq &  \,  \varepsilon R^2\mathcal{I}_1 + \frac{C}{R^{2}}\left(\mathcal{I}_2+\mathcal{I}_3  \right) +\frac{C}{R^{2-36\beta}}\mathcal{I}_4+ CR^2+ \varepsilon\mathcal{I}_4 ' + \frac{C_\varepsilon}{R^2} \mathcal{I}_3' \nonumber \\
\leq & \,   \varepsilon R^2\mathcal{I}_1 + \frac{C}{R^{2-36\beta}}\left(\mathcal{I}_2+\mathcal{I}_3'  +
\mathcal{I}_4'\right) + \varepsilon\mathcal{I}_4 ' + \frac{C_\varepsilon}{R^2} \mathcal{I}_3'+CR^2 \, ,
\end{align}
where we used \eqref{eqii3}, \eqref{eqi7} and the following Young inequality
$$
\mathrm{Re}(f_\alpha\eta_{\bar{\alpha}})\vert\partial f\vert^2 e^{-2\beta f}\eta^{s-3} \leq \vert\partial f\vert\vert\partial\eta\vert\vert\partial f\vert^2 e^{-2\beta f}\eta^{s-3}\leq \varepsilon \vert\partial f\vert^4 e^{-2\beta f} \eta^{s-2} + C_\varepsilon \vert\partial f\vert^2e^{-2\beta f} \vert\partial\eta\vert^2 \eta^{s-4}\, .
$$
Hence, by choosing $\beta, \varepsilon$ small enough and $R$ large enough in \eqref{eqi6} from \eqref{eqi7} we have
\begin{equation}\label{3.15}
\mathcal{I}_4\leq \mathcal{I}_4' \leq  2\varepsilon R^2\mathcal{I}_1 +\frac{C}{R^{2-36\beta}}\left(\mathcal{I}_2+\mathcal{I}_3' \right)+CR^2\, .
\end{equation}
Coming back to \eqref{eqi2} we obtain, by using \eqref{3.15},
\begin{align}\nonumber
\mathcal{I}_2\leq&C(\eps+\beta) R^2 \mathcal{I}_1+\frac{C(1+\beta)}{R^2}\mathcal{I}_3+C\mathcal{I}_4 \\
\leq& C \eps  R^2\mathcal{I}_1 +\frac{C}{R^2}\mathcal{I}_3'+ 2\varepsilon R^2\mathcal{I}_1 +\frac{C}{R^{2-36\beta}}\left(\mathcal{I}_2+\mathcal{I}_3' \right) +CR^2\nonumber \\
\leq & C \eps  R^2\mathcal{I}_1  +\frac{C}{R^{2-36\beta}}\left(\mathcal{I}_2+\mathcal{I}_3' \right)+CR^2\nonumber \, ,
\end{align}
i.e.  for $R$ large enough and $\beta$ small enough,
\begin{equation}\label{eqfun}
\mathcal{I}_2\leq C \eps  R^2\mathcal{I}_1  +\frac{C}{R^{2-36\beta}}\mathcal{I}_3'  +CR^2\, .
\end{equation}
By using \eqref{eqi8} in \eqref{eqfun} we get
\begin{equation}\label{eqfunbis}
\mathcal{I}_2\leq C \eps  R^2\mathcal{I}_1  +C R^{40\beta}+CR^2\leq   \eps  R^2\mathcal{I}_1 +CR^2\, .
\end{equation}
Finally, from \eqref{eqifinal} and \eqref{eqfunbis} we obtain
\begin{equation*}
\mathcal{I}_1\leq C \eps  \mathcal{I}_1  +C  \, ,
\end{equation*}
i.e. for $\varepsilon$ small enough
\begin{equation*}
\mathcal{I}_1=\int_{B_R} \mathcal{M}\Psi^{-\beta}\eta^s\leq  C,
\end{equation*}
and, from \eqref{eqfunbis},
$$
\mathcal{I}_2=\int_{B_R}|g|^2\Psi^{-\beta}\eta^{s-2}\leq CR^2.
$$
Then
$$
\int_{\mathbb{H}^1} \mathcal{M}\Psi^{-\beta}\leq C.
$$
In particular, from Lemma \ref{lemma0}, it follows that
\begin{align*}
\int_{\mathbb{H}^1} \mathcal{M}\Psi^{-\beta}\eta^s&\leq  \frac{C}{R} \left(\int_{\text{supp}|\partial\eta|}\mathcal{M}\Psi^{-\beta}\eta^s\right)^{1/2}\left(\int_{\text{supp}|\partial\eta|}|g|^2\Psi^{-\beta}\eta^{s-2}\right)^{1/2}\\
&\leq C\left(\int_{A_R} \mathcal{M}\Psi^{-\beta}\right)^{1/2}\longrightarrow 0
\end{align*}
as $R\to\infty$, i.e.
$$
\mathcal{M}\equiv 0.
$$
The conclusion follows arguing as in \cite[Section 3]{JL}.

\

\section{Proof of Theorem \ref{teo2}}\label{proof_teo2}

In this section $n\geq 2$ and given $R>1$, we choose, as in the previous section, a real cut-off function $\eta$ such that $\eta\equiv 1$ in $B_{R/2}$, $\eta\equiv 0$ in $B_R^c$ and $\vert \partial\eta\vert\leq \frac{c}{R}$ in $A_R:=B_R\setminus B_{R/2}$, and $s>4$.

From Lemma \ref{lemma1} we have
\begin{align}\label{eq132}\nonumber
\int_{\mathbb{H}^1} \eta^s\mathcal{M}\leq&  C \int_{\mathbb{H}^1} e^{2(n-1)f} \vert\partial\eta\vert^2 \eta^{s-2}\vert g\vert^2 \\
\leq& \dfrac{C}{R^2} \int_{A_R}e^{2(n-1)f}  \eta^{s-2} g\bar{g} \\ \nonumber
=& -\dfrac{C}{R^2} \int_{A_R}e^{2(n-1)f}  \eta^{s-2}\mathrm{Re}\left( f_{\alpha\bar{\alpha}}\bar{g}\right)
\end{align}
where we used \eqref{eq_g}. Integrating by parts we get
\begin{align*}
-\dfrac{C}{R^2} \int_{A_R} e^{2(n-1)f} \eta^{s-2} \mathrm{Re}\left(f_{\alpha\bar{\alpha}}\bar{g} \right)&= \dfrac{C}{R^2}  \int_{A_R} e^{2(n-1)f} \mathrm{Re}\left(\eta^{s-2} f_\alpha \bar{g}_{\bar{\alpha}}\right) \\
&\quad+ \dfrac{C(s-2)}{R^2}  \int_{A_R} e^{2(n-1)f} \mathrm{Re}\left(\eta^{s-3}f_\alpha\eta_{\bar{\alpha}}\bar{g}\right)\\
&\quad+\frac{C}{R^2} \int_{A_R} e^{2(n-1)f} \eta^{s-2}|\partial f|^2\mathrm{Re}(\bar g)\\
&=: \frac{C}{R^2} \left( \mathcal{J}_1 + \mathcal{J}_2+\mathcal{J}_3\right) \,.
\end{align*}
In particular, we have
\begin{equation}\label{eq122}
\int_{A_R} e^{2(n-1)f}  \eta^{s-2} |g|^2\leq C(\mathcal{J}_1+\mathcal{J}_2+\mathcal{J}_3).
\end{equation}
From \eqref{gbarra}, using Young and Cauchy-Schwartz inequalities, we obtain
\begin{align*}
\mathcal{J}_1&=\int_{A_R}  e^{2(n-1)f} \mathrm{Re}\left(\eta^{s-2} f_\alpha \bar{g}_{\bar{\alpha}}\right)\\
&=\int_{A_R} e^{2(n-1)f} \mathrm{Re}\left(\eta^{s-2} f_\alpha (D_{\bar{\alpha}}+E_{\bar{\alpha}}-G_{\bar{\alpha}})\right)+ 2\int_{A_R} e^{2(n-1)f}  \mathrm{Re}\left(\eta^{s-2}\vert\partial f\vert^2  \bar{g}\right) \\
&\leq  \frac{1}{2}\int_{A_R} e^{2(n-1)f} \eta^{s}\vert D_{\bar{\alpha}}+E_{\bar{\alpha}}-G_{\bar{\alpha}}\vert^2 + \frac{1}{2}\int_{A_R} e^{2(n-1)f} \eta^{s-4} \vert\partial f\vert^2 \\
&\quad+ 2 \int_{A_R} e^{2(n-1)f}  \left( \vert\partial f\vert^4 + e^{2f} \vert\partial f\vert^2\right)\eta^{s-2} \\
&\leq  \frac{1}{2}\int_{A_R} e^{2(n-1)f} \eta^{s}\left(\vert D_{\bar{\alpha}}+G_{\bar{\alpha}}\vert +\vert E_{\bar{\alpha}}-G_{\bar{\alpha}}\vert+ \vert G_{\bar{\alpha}}\vert\right)^2 + \frac{1}{2}\int_{A_R} e^{2(n-1)f} \eta^{s-4} \vert\partial f\vert^2 \\
&\quad+ \frac{1}{2} \int_{A_R}  e^{2(n+1)f} \eta^{s-2}  + 4\int_{A_R}  e^{2(n-1)f} \vert\partial f\vert^4 \eta^{s-2} \\
&\leq  \frac{3}{2}\int_{A_R} \eta^{s}\mathcal{M}  + \frac{1}{2}\int_{A_R} e^{2(n-1)f} \eta^{s-4} \vert\partial f\vert^2 \\
&\quad+ \frac{1}{2} \int_{A_R}  e^{2(n+1)f} \eta^{s-2}  + 4\int_{A_R} e^{2(n-1)f}  \vert\partial f\vert^4 \eta^{s-2}
\end{align*}
where we also used the definitions of $g$ in \eqref{g} and of $\mathcal{M}$ in Proposition \ref{prop_JL}.

Now, we prove that
\begin{equation}\label{claim2}
\int_{A_R} e^{2(n-1)f} \eta^{s-4} \vert\partial f\vert^2 \leq CR^2\, ,
\end{equation}
indeed, by Proposition \ref{gradest}, we have
$$
\sup_{A_R}|\partial f|^2\leq \frac{C}{R^2}.
$$
Moreover, by assumption, in $A_R$ we have
$$
e^{2(n-1)f}=u^{\frac{2(n-1)}{n}}=u^{\frac{2(Q-4)}{Q-2}}\leq \frac{C}{R^{Q-4}}.
$$
Therefore, using the volume estimate \eqref{volumi},
$$
\int_{A_R} e^{2(n-1)f} \eta^{s-4} \vert\partial f\vert^2 \leq \frac{C}{R^{Q-2}}|B_R|\leq CR^2\, ,
$$
and \eqref{claim2} follows. Then from \eqref{claim2} we have
\begin{equation}\label{eq162}
\int_{A_R} e^{2(n-1)f} \eta^{s-2} \vert\partial f\vert^4 \leq \frac{C}{R^2}\int_{A_R} e^{2(n-1)f} \eta^{s-2} \vert\partial f\vert^2\leq C\, ,
\end{equation}
Thus, from the previous computation, \eqref{claim2} and \eqref{eq162}, we obtain
\begin{align*}
\mathcal{J}_1=\int_{A_R}e^{2(n-1)f}  \mathrm{Re}\left(\eta^{s-2} f_\alpha \bar{g}_{\bar{\alpha}}\right) \leq &C\int_{A_R} \eta^{s}\mathcal{M}  + C R^2 + \frac{1}{2} \int_{A_R}  e^{2(n+1)f} \eta^{s-2}  + C\, .
\end{align*}
On the other hand
\begin{align*}
\mathcal{J}_2= \int_{A_R}e^{2(n-1)f}  \mathrm{Re}\left(\eta^{s-3}f_\alpha\eta_{\bar{\alpha}}\bar{g}\right) \leq& \frac12 \int_{A_R} e^{2(n-1)f} \eta^{s-4} \vert\partial f\vert^2 + \frac12 \int_{A_R} e^{2(n-1)f} \eta^{s-2} \vert g\vert^2 \vert\partial\eta\vert^2  \\
\leq & C R^2 + \frac{C}{R^2} \int_{A_R}e^{2(n-1)f}  \eta^{s-2} \vert g\vert^2 \,
\end{align*}
and, for every $\delta>0$ small enough,
\begin{align*}
\mathcal{J}_3= \int_{A_R} e^{2(n-1)f} \eta^{s-2}|\partial f|^2\mathrm{Re}(\bar g)\leq& \frac{1}{4\delta} \int_{A_R} e^{2(n-1)f} \eta^{s-2} \vert\partial f\vert^4 + \delta \int_{A_R} e^{2(n-1)f} \eta^{s-2} \vert g\vert^2  \\
\leq & \frac{C}{\delta} + \delta\int_{A_R}e^{2(n-1)f}  \eta^{s-2} \vert g\vert^2 \,
\end{align*}
Thus, from \eqref{eq122} and \eqref{claim2}, we obtain
\begin{align*}
 \int_{A_R}e^{2(n-1)f} \eta^{s-2} \vert g\vert^2 &\leq C\int_{A_R} \eta^{s}\mathcal{M}  +\frac{C}{\delta}R^2+\frac12\int_{A_R}e^{2(n+1)f}\eta^{s-2}+\left(\frac{C}{R^2}+\delta\right) \int_{A_R}e^{2(n-1)f} \eta^{s-2} \vert g\vert^2\,.
\end{align*}
For $\delta$ small enough and $R$ large enough, we get
\begin{align*}
 \int_{A_R}e^{2(n-1)f} \eta^{s-2} \vert g\vert^2 &\leq C\int_{A_R} \eta^{s}\mathcal{M}  + C R^2+\frac{2}{3}\int_{A_R}  e^{2(n+1)f} \eta^{s-2} \,.
\end{align*}
Recalling \eqref{mod_g}, we have
%$$
%|g|^2=|\partial f|^4+2|\partial f|^2e^{2f}+e^{4f}+f_0^2,
%$$
\begin{align*}
 \int_{A_R}e^{2(n-1)f} \eta^{s-2} \left(|\partial f|^4+2|\partial f|^2e^{2f}+\frac{1}{3}e^{4f}+f_0^2\right)&\leq C\int_{A_R} \eta^{s}\mathcal{M}  + C R^2
\end{align*}
and therefore
\begin{equation}\label{eq152}
 \int_{A_R}e^{2(n-1)f}\eta^{s-2} \vert g\vert^2\leq C\int_{A_R} \eta^{s}\mathcal{M}  + C R^2
\,.
\end{equation}
Going back to \eqref{eq132}, for $R$ large enough, we obtain
\begin{equation}\label{eq142}
 \int_{\H^n}\mathcal M\eta^s\leq \frac{C}{R^2}\int_{A_R}e^{2(n-1)f}\eta^{s-2} \vert g\vert^2 \leq C\,.
\end{equation}
Hence,
$$
 \int_{\H^n}\mathcal M\leq C\quad\text{and}\quad  \int_{A_R}e^{2(n-1)f}\eta^{s-2} \vert g\vert^2\leq CR^2.
$$
In particular, from Lemma \ref{lemma1}, it follows that
\begin{align*}
\int_{\mathbb{H}^n} \mathcal{M}\eta^s&\leq  \frac{C}{R} \left(\int_{\text{supp}|\partial\eta|}\mathcal{M}\eta^s\right)^{1/2}\left(\int_{\text{supp}|\partial\eta|}e^{2(n-1)f} |g|^2\eta^{s-2}\right)^{1/2}\\
&\leq C\left(\int_{A_R} \mathcal{M}\right)^{1/2}\longrightarrow 0
\end{align*}
as $R\to\infty$, i.e.
$$
\mathcal{M}\equiv 0.
$$
The conclusion follows arguing as in \cite[Section 3]{JL}.

\

\appendix

\section{Gradient estimates}\label{grad_est}

In this section we will prove some gradient estimates on positive solutions of equation \eqref{heis_crit}, that we need in the proof of Theorem \ref{teo2}.

\begin{proposition}\label{gradest}
  Let $n\in\mathbb{N}$ and let $u$ be a positive solution of \eqref{heis_crit}. If
\begin{equation}\label{decay1}
  u(\xi)\leq\frac{C}{1+|\xi|^\frac{Q-2}{2}}\quad\forall\xi\in \H^n,
\end{equation}
then there exists $C>0$ such that for every large enough $R>0$
\begin{equation}\label{gradest1}
\sup_{B_{2R}\setminus B_R}\frac{|\partial u|}{u}\leq \frac{C}{R} .
\end{equation}
\end{proposition}

\begin{remark}
  We explicitly note that, by Proposition \ref{gradest}, if $u$ is a positive solution of \eqref{heis_crit} satisfying \eqref{decay1} and if $f=\frac{1}{n}\log u$ then
  $$
  \sup_{B_{2R}\setminus B_R}|\partial f|^2\leq \frac{C}{R^2}
  $$
  for some $C>0$ and every large enough $R>0$.
\end{remark}

\begin{proof}[Proof of Proposition \ref{gradest}]
  The proof follows arguments similar in spirit to \cite{HeZa}, where the authors consider harmonic functions for the sub-Laplacian. Let $f=\frac{1}{n}\log u$ and let
  $$
  \nabla f:=f_{\bar{\alpha}}Z_\alpha+f_\alpha Z_{\bar{\alpha}}.
  $$
  Then we have, see also \eqref{eq_f},
  \begin{align*}
    |\nabla f|^2 & = 2 f_{\bar{\alpha}}f_\alpha=2|\partial f|^2, \\
      \Delta_{\mathbb{H}^n} f & = -n |\nabla f|^2 - 2n e^{2f},\\
      \nabla \Delta_{\mathbb{H}^n} f & = -n\nabla \Delta_{\mathbb{H}^n} f -4ne^{2f}\nabla f.
  \end{align*}
Now if $\eta$ is any nonnegative smooth cutoff function on $\mathbb{H}^n$ we define for every $t\in(0,1]$
 $$
 F:=t(|\nabla f|^2+\gamma t\eta f_0^2)\, ,
 $$
 where $\gamma>0$ is to be chosen later.
 Then we have
 \begin{equation}\label{eq3}
 \Delta_{\mathbb{H}^n}F=t\Delta_{\mathbb{H}^n}|\nabla f|^2+\gamma t^2\eta \Delta_{\mathbb{H}^n}f_0^2+2\gamma t^2\langle\nabla\eta,\nabla f_0^2\rangle+\gamma t^2f_0^2\Delta_{\mathbb{H}^n}\eta.
 \end{equation}
We recall the following Bochner Formula for the sub-Laplacian on the Heisenberg group (see \cite{CKLT} and also \cite[Lemma 2.3]{HeZa}): for every function $\nu:\mathbb{H}^n\rightarrow(0,+\infty)$ we have
\begin{equation}\label{Bochner}
  \Delta_{\mathbb{H}^n}|\nabla f|^2\geq\frac{1}{n}(\Delta_{\mathbb{H}^n} f)^2+nf_0^2+2\langle\nabla f,\nabla \Delta_{\mathbb{H}^n} f\rangle -\frac{2}{\nu}|\nabla f|^2-2\nu|\nabla f_0|^2.
\end{equation}
Moreover for every $\varepsilon\in(0,\frac 14)$ we have
 \begin{align}
   (\Delta_{\mathbb{H}^n} f)^2&=\left(-\frac{\varepsilon}{t}F-(n-\varepsilon)|\nabla f|^2-2ne^{2f}+\varepsilon\gamma t\eta f_0^2\right)^2 \nonumber\\
  \label{eq2} &\geq \frac{\varepsilon^2}{t^2}F^2+2\frac{\varepsilon}{t}(n-\varepsilon)F|\nabla f|^2-2\varepsilon^2\gamma f_0^2F\eta,
 \end{align}
 and
\begin{equation}\label{eq1}
  \Delta_{\mathbb{H}^n}f_0^2=-8ne^{2f}f_0^2-2n\langle\nabla f,\nabla f_0^2\rangle+2|\nabla f_0|^2
\end{equation}
By \eqref{eq3}-\eqref{eq1} we obtain
  \begin{equation*}
    \begin{aligned}
    \Delta_{\mathbb{H}^n} F&\geq\frac{\varepsilon^2}{nt}F^2+2\varepsilon(1-\tfrac{\varepsilon}{n})F|\nabla f|^2-\frac{2}{n}t\varepsilon^2\gamma\eta f_0^2F+ntf_0^2\\
    &\,\,\,\,\,\, -8nte^{2f}|\nabla f|^2-2nt\langle\nabla|\nabla f|^2,\nabla f\rangle-\frac{2t}{\nu}|\nabla f|^2-2\nu t|\nabla f_0|^2\\
    &\,\,\,\,\,\, +t^2\gamma \eta(-8ne^{2f}f_0^2-2n\langle\nabla f,\nabla f_0^2\rangle +2|\nabla f_0|^2)\\
    &\,\,\,\,\,\, +4t^2\gamma f_0\langle\nabla\eta,\nabla f_0\rangle+t^2\gamma f_0^2\Delta_{\mathbb{H}^n} \eta.
    \end{aligned}
  \end{equation*}
  Since
  $$
  \langle\nabla|\nabla f|^2,\nabla f\rangle=\langle\tfrac{1}{t}\nabla F-t\gamma\eta\nabla f_0^2-t\gamma f_0^2\nabla\eta,\nabla f\rangle
  $$
  then we have
  \begin{equation*}
    \begin{aligned}
    \Delta_{\mathbb{H}^n} F&\geq\frac{\varepsilon^2}{nt}F^2+2\varepsilon(1-\tfrac{\varepsilon}{n})F|\nabla f|^2-\frac{2}{n}t\varepsilon^2\gamma\eta f_0^2F+ntf_0^2\\
    &\,\,\,\,\,\, -8nte^{2f}|\nabla f|^2-2n\langle\nabla F,\nabla f\rangle+2nt^2\gamma f_0^2\langle\nabla\eta,\nabla f\rangle-\frac{2t}{\nu}|\nabla f|^2-2\nu t|\nabla f_0|^2\\
    &\,\,\,\,\,\, +t^2\gamma \eta(-8ne^{2f}f_0^2 +2|\nabla f_0|^2)+4t^2\gamma f_0\langle\nabla\eta,\nabla f_0\rangle+t^2\gamma f_0^2\Delta_{\mathbb{H}^n} \eta.
    \end{aligned}
  \end{equation*}
For every $\xi$ such that $\eta(\xi)\neq0$ there holds
$$
|4t^2\gamma f_0\langle\nabla\eta,\nabla f_0\rangle|\leq t^2\gamma\eta|\nabla f_0|^2+4t^2\gamma\frac{|\nabla\eta|^2}{\eta}f_0^2,
$$
then
  \begin{equation*}
    \begin{aligned}
    \Delta_{\mathbb{H}^n} F&\geq\frac{\varepsilon^2}{nt}F^2+2\varepsilon(1-\tfrac{\varepsilon}{n})F|\nabla f|^2-\frac{2}{n}t\varepsilon^2\gamma\eta f_0^2F+ntf_0^2\\
    &\,\,\,\,\,\, -8nte^{2f}|\nabla f|^2-2n\langle\nabla F,\nabla f\rangle+2nt^2\gamma f_0^2\langle\nabla\eta,\nabla f\rangle-\frac{2t}{\nu}|\nabla f|^2-2\nu t|\nabla f_0|^2\\
    &\,\,\,\,\,\, +t^2\gamma \eta(-8ne^{2f}f_0^2 +2|\nabla f_0|^2)-t^2\gamma\eta|\nabla f_0|^2-4t^2\gamma\frac{|\nabla\eta|^2}{\eta}f_0^2+t^2\gamma f_0^2\Delta_{\mathbb{H}^n} \eta.
    \end{aligned}
  \end{equation*}
For every $\xi$ such that $\eta(\xi)\neq0$ we choose $\nu=\frac{t\gamma\eta}{2}$, then
  \begin{equation*}
    \begin{aligned}
    \Delta_{\mathbb{H}^n} F&\geq\frac{\varepsilon^2}{nt}F^2+2\varepsilon(1-\tfrac{\varepsilon}{n})F|\nabla f|^2-\frac{2}{n}t\varepsilon^2\gamma\eta f_0^2F+ntf_0^2\\
    &\,\,\,\,\,\, -8nte^{2f}|\nabla f|^2-2n\langle\nabla F,\nabla f\rangle+2nt^2\gamma f_0^2\langle\nabla\eta,\nabla f\rangle-\frac{4}{\gamma\eta}|\nabla f|^2\\
    &\,\,\,\,\,\, -8nt^2\gamma \eta e^{2f}f_0^2 -4t^2\gamma\frac{|\nabla\eta|^2}{\eta}f_0^2+t^2\gamma f_0^2\Delta_{\mathbb{H}^n} \eta.
    \end{aligned}
  \end{equation*}
Now note that for every $\xi$ such that $\eta(\xi)\neq0$
  $$
  |2nt^2\gamma f_0^2\langle\nabla\eta,\nabla f\rangle|\leq \varepsilon^2nt\gamma\eta f_0^2 F+\frac{nt^2\gamma}{\varepsilon^2}f_0^2\frac{|\nabla\eta|^2}{\eta},
  $$
  hence
 \begin{equation*}
    \begin{aligned}
    \Delta_{\mathbb{H}^n} F&\geq\frac{\varepsilon^2}{nt}F^2+2\varepsilon(1-\tfrac{\varepsilon}{n})F|\nabla f|^2-\left(n+\tfrac{2}{n}\right)t\varepsilon^2\gamma\eta f_0^2F+ntf_0^2\\
    &\,\,\,\,\,\, -8nte^{2f}|\nabla f|^2-2n\langle\nabla F,\nabla f\rangle-\frac{nt^2\gamma}{\varepsilon^2}\frac{|\nabla\eta|^2}{\eta}f_0^2-\frac{4}{\gamma\eta}|\nabla f|^2\\
    &\,\,\,\,\,\, -8nt^2\gamma \eta e^{2f}f_0^2 -4t^2\gamma\frac{|\nabla\eta|^2}{\eta}f_0^2+t^2\gamma f_0^2\Delta_{\mathbb{H}^n} \eta.
    \end{aligned}
  \end{equation*}
Let $H=\eta F$, then $H\geq0$ and for every $\xi$ such that $H(\xi)\neq0$
\begin{equation*}
    \begin{aligned}
    \Delta_{\mathbb{H}^n} H&\geq F\Delta_{\mathbb{H}^n}\eta +2\langle\nabla\eta,\nabla F\rangle +\eta\bigg[\frac{\varepsilon^2}{nt}F^2+2\varepsilon(1-\tfrac{\varepsilon}{n})F|\nabla f|^2-\left(n+\tfrac{2}{n}\right)t\varepsilon^2\gamma\eta f_0^2F\\
    &\,\,\,\,\,\, +ntf_0^2-8nte^{2f}|\nabla f|^2-2n\langle\nabla F,\nabla f\rangle-\frac{nt^2\gamma}{\varepsilon^2}\frac{|\nabla\eta|^2}{\eta}f_0^2-\frac{4}{\gamma\eta}|\nabla f|^2\\
    &\,\,\,\,\,\, -8nt^2\gamma \eta e^{2f}f_0^2 -4t^2\gamma\frac{|\nabla\eta|^2}{\eta}f_0^2+t^2\gamma f_0^2\Delta_{\mathbb{H}^n} \eta\bigg].
    \end{aligned}
  \end{equation*}
Let $\eta=\eta(|\xi|)$ be a radial, nonnegative, smooth cutoff function such that $\eta=0$ in $B_\frac{R}{2}\cup B^c_{\frac{5}{2}R}$, $\eta=1$ on $A_R=B_{2R}\setminus B_R$ and satisfying
$$
|\nabla\eta|\leq\frac{C}{R}\sqrt{\eta},\quad\Delta_{\mathbb{H}^n}\eta\geq-\frac{C}{R^2}
$$
on $\mathbb{H}^n$, for some $C>0$. Since $u$ satisfies \eqref{decay1} we have
$$
e^{2f}=u^\frac{2}{n}\leq \frac{\hat{C}}{R^2}
$$
in $B_{\frac{5}{2}R}\setminus B_\frac{R}{2}$. Since $t\in(0,1)$ and $0\leq\eta\leq1$, at any $\xi$ such that $H(\xi)>0$, which in particular must satisfy $\xi\in B_{\frac{5}{2}R}\setminus B_\frac{R}{2}$, we thus obtain
\begin{equation*}
    \begin{aligned}
    t\eta\Delta_{\mathbb{H}^n} H&\geq -\frac{Ct}{R^2}H+2t\langle\nabla\eta,\nabla H\rangle-2tF|\nabla\eta|^2 +\frac{\varepsilon^2}{n}H^2\\
    &\,\,\,\,\,\,+t\eta|\nabla f|^2\left(2\varepsilon(1-\tfrac{\varepsilon}{n})H-8ne^{2f}-\frac{4}{\gamma}\right)\\
    &\,\,\,\,\,\,+t^2\eta^2f_0^2\left(n-(n+\tfrac{2}{n})\varepsilon^2\gamma H-8n\gamma e^{2f}-\frac{C\gamma}{\varepsilon^2R^2}\right)\\
    &\,\,\,\,\,\,-2nt\eta\langle \nabla H,\nabla f\rangle+2nt\langle\nabla\eta,\nabla f\rangle H\\
    &\geq \frac{\varepsilon^2}{n}H^2-\frac{C}{\varepsilon R^2}H+2t\langle\nabla H,\nabla\eta-n\eta\nabla f\rangle\\
    &\,\,\,\,\,\,+t\eta|\nabla f|^2\left(2\varepsilon(1-\tfrac{\varepsilon}{n})H-\frac{8n\hat{C}}{R^2}-\frac{4}{\gamma}\right)\\
    &\,\,\,\,\,\,+t^2\eta^2f_0^2\left(n-(n+\tfrac{2}{n})\varepsilon^2\gamma H-\frac{8n\gamma \hat{C}}{R^2}-\frac{C\gamma}{\varepsilon^2R^2}\right)+2nt\langle\nabla\eta,\nabla f\rangle H.
    \end{aligned}
  \end{equation*}
Now we note that
$$
|2nt\langle\nabla\eta,\nabla f\rangle H|\leq\varepsilon tH|\nabla f|^2\eta+\frac{Ct}{\varepsilon} \frac{|\nabla\eta|^2}{\eta} H\leq \varepsilon tH|\nabla f|^2\eta+\frac{C}{\varepsilon R^2}H.
$$
Therefore we have
\begin{equation*}
    \begin{aligned}
    t\eta\Delta_{\mathbb{H}^n} H-2t\langle\nabla H,\nabla\eta-n\eta\nabla f\rangle&\geq \frac{\varepsilon^2}{n}H^2-\frac{C}{\varepsilon R^2}H\\
    &\,\,\,\,\,\,+t\eta|\nabla f|^2\left(\varepsilon(1-\tfrac{2\varepsilon}{n})H-\frac{8n\hat{C}}{R^2}-\frac{4}{\gamma}\right)\\
    &\,\,\,\,\,\,+t^2\eta^2f_0^2\left(n-(n+\tfrac{2}{n})\varepsilon^2\gamma H-\frac{8n\gamma \hat{C}}{R^2}-\frac{C\gamma}{\varepsilon^2R^2}\right).
    \end{aligned}
  \end{equation*}
Choosing $\gamma=\varepsilon^\frac{5}{2}R^2$ we obtain
\begin{equation*}
    \begin{aligned}
    t\eta\Delta_{\mathbb{H}^n} H-2t\langle\nabla H,\nabla\eta-n\eta\nabla f\rangle&\geq \left(\varepsilon^2H-\frac{C_0}{\varepsilon R^2}\right)\frac{H}{n}\\
    &\,\,\,\,\,\,+t\eta|\nabla f|^2\left(\frac{\varepsilon}{2}H-\frac{C_0}{\varepsilon^\frac{5}{2}R^2}\right)\\
    &\,\,\,\,\,\,+t^2\eta^2f_0^2\left(n-(n+\tfrac{2}{n})\varepsilon^\frac{9}{2}R^2 H-C_0\varepsilon^\frac{1}{2}\right),
    \end{aligned}
  \end{equation*}
  for some $C_0>0$. $H$ achieves its positive maximum on $B_{\frac{5}{2}R}\setminus B_{\frac{R}{2}}$ at some point $p\in B_{\frac{5}{2}R}\setminus B_{\frac{R}{2}}$. If $H_*=H(p)$ and $\eta_*=\eta(p)$ then
  \begin{equation}\label{eq4}
  \begin{aligned}
  0\geq &\left(\varepsilon^2H_*-\frac{C_0}{\varepsilon R^2}\right)\frac{H_*}{n}+t\eta_*|\nabla f|^2\left(\frac{\varepsilon}{2}H_*-\frac{C_0}{\varepsilon^\frac{5}{2}R^2}\right)\\ &+t^2\eta^2_*f_0^2\left(n-(n+\tfrac{2}{n})\varepsilon^\frac{9}{2}R^2 H_*-C_0\varepsilon^\frac{1}{2}\right).
  \end{aligned}
  \end{equation}
  We claim that there exist $C>0$, $R_0>0$ such that at $t=1$
  $$
  H_*<\frac{2C}{R}
  $$
  for every $R\geq R_0$. If not, by contradiction, for every $C>0$, $R_0>0$ there exists $R_1\geq R_0$ such that at $t=1$
  $$
  H_*\geq\frac{2C}{R_1^2}.
  $$
  Now note that $H$ (and thus also $H_*$) are continuous functions in $t\in[0,1]$, satisfying $H_*=0$ for $t=0$ and $H_*\geq\frac{2C}{R_1^2}$ for $t=1$. Then for every $C>0$, $R_0>0$ there exist $R_1\geq R_0$, $t_1\in(0,1)$ such that at $t=t_1$
  $$
  H_*=\frac{C}{R^2_1}.
  $$
  We choose $C=\frac{1}{\varepsilon^4}$, $R_0=1$, then at $t=t_1$, $R=R_1$ we have from \eqref{eq4}
  \begin{equation*}
  0\geq \left(\frac{1}{\varepsilon^2R_1^2}-\frac{C_0}{\varepsilon R_1^2}\right)\frac{H_*}{n} +t\eta_*|\nabla f|^2\left(\frac{1}{2\varepsilon^3R_1^2}-\frac{C_0}{\varepsilon^\frac{5}{2}R_1^2}\right)+t^2\eta^2_*f_0^2\left(n-(n+\tfrac{2}{n}+C_0)\varepsilon^\frac{1}{2}\right)>0
  \end{equation*}
  if we choose $\varepsilon>0$ small enough, thus reaching a contradiction. We conclude that there exist $C,R_0>0$ such that at $t=1$
  $$
  H_*=\max_{B_{\frac{5}{2}R}\setminus B_{\frac{R}{2}}} H\leq\frac{C}{R^2}
  $$
  for every $R\geq R_0$, which in turn implies that
  $$
  \sup_{B_{2R}\setminus B_R}|\nabla f|^2\leq\max_{B_{\frac{5}{2}R}\setminus B_{\frac{R}{2}}} H\leq\frac{C}{R^2}
  $$
  for every $R\geq R_0$. Recalling the definition of $f$, we conclude that \eqref{gradest1} holds.
\end{proof}

We conclude the section with a second gradient estimate, that may have some independent interest. Since we do not need this estimate in the proof of our other results, and since its proof follows using similar arguments as those employed in Proposition \ref{gradest}, we state the result without proof.

\begin{proposition}\label{gradest2}
  Let $n\in\mathbb{N}$ and let $u$ be a positive solution of \eqref{heis_crit}. If $u$ is bounded on $\mathbb{H}^n$, then $\frac{|\partial u|}{u}$ is also bounded on $\mathbb{H}^n$.
\end{proposition}

\

\

\begin{ackn}
\noindent
The first author is member of the GNSAGA,  Gruppo Nazionale per le Strutture Algebriche, Geometriche e le loro Applicazioni of INdAM. The third and the fourth authors are members of GNAMPA, Gruppo Nazionale per l'Analisi Matematica, la Probabilit\`a e le loro Applicazioni of INdAM.

Moreover, the first and the fourth authors are partially supported by the project PRIN 2022 ``Differential-geometric aspects of manifolds via Global Analysis'', the second author is partially supported by NSF Grant DMS-2247410, while the third author has been partially supported by the project PRIN 2022 ``Geometric-Analytic Methods for PDEs and Applications''.
\end{ackn}

\

\noindent{\bf Data availability statement}

\noindent Data sharing not applicable to this article as no datasets were generated or analysed during the current study.

\

\

\

\

\

\end{document}